\documentclass[12pt]{article}%
\usepackage{amsmath}
\usepackage{amsfonts}
\usepackage{amssymb}
\usepackage{graphicx}
\usepackage{subfig}%
\usepackage{epstopdf}
\providecommand{\U}[1]{\protect\rule{.1in}{.1in}}
\newtheorem{theorem}{Theorem}

\newtheorem{example}[theorem]{Example}

\newtheorem{lemma}[theorem]{Lemma}

\newtheorem{remark}[theorem]{Experiment}

\newenvironment{proof}[1][Proof]{\noindent\textbf{#1.} }{\ \rule{0.5em}{0.5em}}
\begin{document}

\title{Efficient vectors in priority setting methodology }
\author{Susana Furtado\thanks{Email: sbf@fep.up.pt The work of this author was
supported by FCT- Funda\c{c}\~{a}o para a Ci\^{e}ncia e Tecnologia, under
project UIDB/04721/2020.} \thanks{Corresponding author.}\\CEAFEL and Faculdade de Economia \\Universidade do Porto\\Rua Dr. Roberto Frias\\4200-464 Porto, Portugal
\and Charles R. Johnson \thanks{Email: crjohn@wm.edu. The work of this author was
supported in part by National Science Foundation grant DMS-0751964.}\\Department of Mathematics\\College of William and Mary\\Williamsburg, VA 23187-8795.}
\maketitle

\begin{abstract}
The Analytic Hierarchy Process (AHP) is a much discussed method in ranking
business alternatives based on empirical and judgemental information. We focus
here upon the key component of deducing efficient vectors for a reciprocal
matrix of pair-wise comparisons. It is not known how to produce all efficient
vectors. It has been shown that the entry-wise geometric mean of all columns
is efficient for any reciprocal matrix. Here, by combining some new basic
observations with some known theory, we 1) give a method for inductively
generating large collections of efficient vectors, and 2) show that the
entry-wise geometric mean of any collection of distinct columns of a
reciprocal matrix is efficient. Based on numerical simulations, there seems to
be no clear advantage over either the geometric mean of all columns or the
right Perron eigenvector of a reciprocal matrix.

\end{abstract}



\bigskip

\section{Introduction\label{s1}}

A method used in decision-making and frequently discussed in the literature is
the Analytic Hierarchy Process (AHP), suggested by Saaty \cite{saaty1977,
Saaty}. Several works since then have developed and discussed many aspects of
the method. See the surveys \cite{choo, is, zeleny}. A key element of the
method is the notion of \emph{pair-wise comparison (PC) matrix}. An $n$-by-$n$
positive matrix $A=[a_{ij}]$ is called a PC matrix if, for all $1\leq i,j\leq
n,$%

\[
a_{ji}=\frac{1}{a_{ij}}.
\]
Each diagonal entry of a PC matrix is $1.$ We refer to the set of all such
matrices as $\mathcal{PC}_{n}$. Often, we refer to these matrices as
\emph{reciprocal matrices}, as do other authors.

The $i,j$ entry of a reciprocal matrix is viewed as a pair-wise ratio
comparison between alternatives $i$ and $j,$ and the intent is to deduce an
ordering of the alternatives from it. If the reciprocal matrix is
\emph{consistent} (\emph{transitive}): $a_{ij}a_{jk}=a_{ik},$ for all triples
$i,j,k,$ there is a unique natural cardinal ordering, given by the relative
magnitudes of the entries in any column. However, in human judgements
consistency is unlikely. Inconsistency can also be an inherent feature of
objective datasets \cite{bozoki2016, chao, csato2013, petroczy2021,
petrocsato}. Then, there will be many vectors that might be deduced from a
reciprocal matrix $A$. Let $w$ be a positive $n$-vector and $w^{(-T)}$ the
transpose of its component-wise inverse. We may try to approximate $A$ by the
consistent matrix $W=ww^{(-T)},$ i.e., we wish to choose $w$ so that $W-A$ is
small in some sense. We say that $w$ is \emph{efficient} for $A$ if, for any
other positive vector $v$ and corresponding consistent matrix $V=vv^{(-T)},$
the entry-wise inequality $\left\vert V-A\right\vert \leq\left\vert
W-A\right\vert $ implies that $v$ and $w$ are proportional. (It follows from
Lemma \ref{lprop} that we give later that this definition is equivalent to
that of other authors for the notion of efficiency \cite{blanq2006,
european}). Clearly, a consistent approximation to a reciprocal matrix $A$
should be based upon a vector efficient for $A.$ If $A\ $is not itself
consistent, the set $\mathcal{E}(A)$ of efficient vectors for $A$ will include
many vectors not proportional to each other. This set is, however, at least
connected \cite{blanq2006}, but, in general, it is difficult to determine the
entire set $\mathcal{E}(A).$ For simplicity, we projectively view proportional
efficient vectors as the same, as they produce the same consistent matrix.
Several methods to study when a vector is efficient were developed and
algorithms to improve an inefficient vector have been provided (see \cite{anh,
blanq2006,bozoki2014, european, FerFur, Fu22} and the references therein).

Despite some criticism \cite{dyer0, dyer, johns, saaty2003}, one of the most
used methods to approximate a reciprocal matrix $A$ by a consistent matrix is
the one proposed by Saaty \cite{saaty1977, Saaty}, in which the consistent
matrix is based upon the right Perron eigenvector of $A$, a positive
eigenvector associated with the spectral radius of $A$ \textrm{\cite{HJ}.} The
efficiency of the Perron eigenvector for certain classes of reciprocal
matrices has been shown \textrm{\cite{p6, p2, FerFur}, }though examples of
reciprocal matrices for which this vector is inefficient are also known
\cite{baj, blanq2006, bozoki2014}. Another method to approximate $A$ by a
consistent matrix is based upon the geometric mean of all columns of $A,$
which is known to be an efficient vector for $A$ \textrm{\cite{blanq2006}}.
Many other proposals for approximating $A$ by a consistent matrix have been
made in the literature (for comparisons of different methods see, for example,
\cite{baj, choo, dij, fichtner86, golany, Kula}).

Before summarizing what we do here, we mention some more notation and
terminology. The Hadamard (or entry-wise) product of two vectors (of the same
size) or matrices (of the same dimension) is denoted by $\circ.$ For example,
if $A,B\in\mathcal{PC}_{n}$ then $A\circ B\in\mathcal{PC}_{n}$, and,
similarly, the $n$-by-$n$ consistent matrices are closed under the Hadamard
product. We use superscripts in parentheses to denote an exponent applied to
all entries of a vector or a matrix. For example
\[
\left(  u_{1}\circ u_{2}\circ\cdots\circ u_{k}\right)  ^{\left(  \frac{^{1}%
}{k}\right)  }%
\]
is the (Hadamard) geometric mean of positive vectors $u_{1},\ldots,u_{k}$ of
the same size. This column geometric mean is what is called the row geometric
mean for instance in \textrm{\cite{blanq2006}.}

For an $n$-by-$n$ matrix $A=[a_{ij}],$ we partition $A$ by columns as
$A=\left[  a_{1},\text{ }a_{2},\ldots,\text{\thinspace}a_{n}\right]  .$ The
principal submatrix determined by deleting (by retaining) the rows and columns
indexed by a subset $K\subseteq\{1,\ldots,n\}$ is denoted by $A(K)$ $(A[K]);$
we abbreviate $A(\{i\})$ as $A(i).$ Note that if $A$ is reciprocal
(consistent) then so is $A(i).$

In Section \ref{s2} we give some (mostly known) background that we will use
and make some related observations. In particular, we present the relationship
between efficiency and strong connectivity of a certain digraph and state the
efficiency of the Hadamard geometric mean of all the columns of a reciprocal
matrix. In Section \ref{s3} we give some (mostly new) additional background
that will also be helpful. In Section \ref{s4} we show explicitly how to
extend efficient vectors for $A(i)$ to efficient vectors for the reciprocal
matrix $A.$ This leads to an algorithm initiated by any $A[\{i,j\}]$, $i\neq
j,$ to produce a subset of $\mathcal{E}(A).$ This subset may not be all of
$\mathcal{E}(A)$ as truncation of an efficient vector for $A$ may not give one
for the corresponding principal submatrix. And we may get different subsets by
starting with different $i,j.$ In Section \ref{s5} we study the relationship
between efficient vectors for a reciprocal matrix $A$ and its columns. As
mentioned, any column of a consistent matrix generates that consistent matrix
and, so, is efficient for it. Similarly, any column of a reciprocal matrix is
efficient for it (Lemma \ref{lcol}), as is the geometric mean of any subset of
the columns (Theorem \ref{thcol}). In Section \ref{numerical}, we study
numerically, using different measures, the performance of these efficient
vectors in approximating $A$ by a consistent matrix and compare them, from
this point of view, with the Perron eigenvector (in cases in which it is
efficient). It will be clear that the geometric mean of all columns can be
significantly outperformed by the geometric mean of other collections of
columns. We also show by example that $\mathcal{E}(A)$ is not closed under
geometric mean (Section \ref{s5}). Finally, in Section \ref{s6} we give some conclusions.

\section{Technical Background \label{s2}}

We start with some known results that are relevant for this work. First, it is
important to know how $\mathcal{E}(A)$ changes when $A$ is subjected to either
a positive diagonal similarity or a permutation similarity, or both (a
monomial similarity).

\begin{lemma}
\textrm{\cite{CFF, Fu22}}\label{lsim} Suppose that $A\in\mathcal{PC}_{n}$ and
$w\in\mathcal{E}(A).$ Then, if $D$ is a positive diagonal matrix ($P$ is a
permutation matrix), then $DAD^{-1}\in\mathcal{PC}_{n}$ and $Dw\in
\mathcal{E}(DAD^{-1})$ ($PAP^{T}\in\mathcal{PC}_{n}$ and $Pw\in\mathcal{E}%
(PAP^{T})$).
\end{lemma}

Next we define a directed graph (digraph) associated with a matrix
$A\in\mathcal{PC}_{n}$ and a positive $n$-vector $w,$ which is helpful in
studying the efficiency of $w$ for $A.$ For $w=\left[
\begin{array}
[c]{ccc}%
w_{1} & \cdots & w_{n}%
\end{array}
\right]  ^{T}$, we denote by $G(A,w)$ the directed graph (digraph) whose
vertex set is $\{1,\ldots,n\}$ and whose directed edge set is%
\[
\{i\rightarrow j:\frac{w_{i}}{w_{j}}\geq a_{ij}\text{, }i\neq j\}.
\]

In \cite{blanq2006} the authors proved that the efficiency of $w$ can be
determined from $G(A,w)$.

\begin{theorem}
\textrm{\cite{blanq2006}}\label{blanq} Let $A\in\mathcal{PC}_{n}$. A positive
$n$-vector $w$ is efficient for $A$ if and only if $G(A,w)$ is a strongly
connected digraph, that is, for all pairs of vertices $i,j,$ with $i\neq j,$
there is a directed path from $i$ to $j$ in $G(A,w)$.
\end{theorem}

Recall \textrm{\cite{HJ}} that $G(A,w)$ is strongly connected if and only if
$(I_{n}+L)^{n-1}$ is positive. Here $I_{n}$ is the identity matrix of order
$n$ and $L=[l_{ij}]$ is the adjacency matrix of $G(A,w),$ that is, $l_{ij}=1$
if $i\rightarrow j$ is an edge in $G(A,w),$ and $l_{ij}=0$ otherwise.

\bigskip

In \cite{blanq2006}, it was shown that the geometric mean of all the columns
of a reciprocal matrix $A$ is an efficient vector for $A$\textrm{. }This
result comes from the fact that the geometric mean minimizes the logarithmic
least squares objective function (see also \cite{CW85}).

\begin{theorem}
\label{theogmallcolumns}\textrm{\cite{blanq2006}}If $A\in\mathcal{PC}_{n},$
then
\[
\left(  a_{1}\circ a_{2}\circ\cdots\circ a_{n}\right)  ^{\left(  \frac{^{1}%
}{n}\right)  }\in\mathcal{E}(A).
\]

\end{theorem}

In \cite{CFF}, all the efficient vectors for a simple perturbed consistent
matrix, that is, a reciprocal matrix obtained from a consistent one by
perturbing one entry above the main diagonal and the corresponding reciprocal
entry, were described. Let $Z_{n}(x),$ with $x>0,$ be the matrix in
$\mathcal{PC}_{n}$ with all entries equal to $1$ except those in positions
$1,n$ and $n,1,$ which are $x$ and $\frac{1}{x},$ respectively.\vspace{0cm}
For any simple perturbed consistent matrix $A\in\mathcal{PC}_{n},$ there is a
positive diagonal matrix $D$ and a permutation matrix $P$ such that%
\[
DPAP^{-1}D^{-1}=Z_{n}(x),
\]
for some $x>0.$ Taking into account Lemma \ref{lsim}, an $n$-vector $w$ is
efficient for $A$ if and only if $DPw$ is efficient for $Z_{n}(x).$ For this
reason, we focused on the description of the efficient vectors for $Z_{n}(x),$
as the efficient vectors for a general simple perturbed consistent matrix can
be obtained from them using Lemma \ref{lsim}.

\begin{theorem}
\cite{CFF}\label{tmain} Let $n\geq3$, $x>0$ and $w=\left[
\begin{array}
[c]{cccc}%
w_{1} & \cdots & w_{n-1} & w_{n}%
\end{array}
\right]  ^{T}$ be a positive vector. Then $w$ is efficient for $Z_{n}(x)$ if
and only if
\[
w_{n}\leq w_{i}\leq w_{1}\leq w_{n}x,\text{ for }i=2,\ldots,n-1,
\]
or%
\[
w_{n}\geq w_{i}\geq w_{1}\geq w_{n}x,\text{ for }i=2,\ldots,n-1.
\]

\end{theorem}

\section{Additional facts on efficiency\label{s3}}

From the following result we may conclude that the definition of efficient
vector given in Section \ref{s1} is equivalent to the one in \cite{blanq2006,
european}.

Here and throughout, if $A\in\mathcal{PC}_{n}$ and $w$ is a positive
$n$-vector, we denote
\[
D(A,w):=ww^{(-T)}-A.
\]
By $|D(A,w)|$ we mean the entry-wise absolute value of $D(A,w).$

\begin{lemma}
\label{lprop}Let $A\in\mathcal{PC}_{n}$ and $v,w$ be positive $n$-vectors.
Then, $\left\vert D(A,w)\right\vert =\left\vert D(A,v)\right\vert $ if and
only if $v$ and $w$ are proportional.
\end{lemma}

\begin{proof}
The "if" claim is trivial. Next we show the "only if" claim. Let $w=\left[
\begin{array}
[c]{ccc}%
w_{1} & \cdots & w_{n}%
\end{array}
\right]  ^{T}$ and $v=\left[
\begin{array}
[c]{ccc}%
v_{1} & \cdots & v_{n}%
\end{array}
\right]  ^{T}.$ Let $i,j\in\{1,\ldots,n\}$ with $i\neq j.$ Suppose that
\begin{equation}
\left\vert a_{ij}-\frac{w_{i}}{w_{j}}\right\vert =\left\vert a_{ij}%
-\frac{v_{i}}{v_{j}}\right\vert \text{ and }\left\vert \frac{1}{a_{ij}}%
-\frac{w_{j}}{w_{i}}\right\vert =\left\vert \frac{1}{a_{ij}}-\frac{v_{j}%
}{v_{i}}\right\vert . \label{mod}%
\end{equation}
If
\[
\left(  a_{ij}-\frac{w_{i}}{w_{j}}\right)  \left(  a_{ij}-\frac{v_{i}}{v_{j}%
}\right)  \geq0,
\]
then (\ref{mod}) implies $\frac{w_{i}}{w_{j}}=\frac{v_{i}}{v_{j}}.$ If
\[
\left(  a_{ij}-\frac{w_{i}}{w_{j}}\right)  \left(  a_{ij}-\frac{v_{i}}{v_{j}%
}\right)  <0
\]
then also
\[
\left(  \frac{1}{a_{ij}}-\frac{w_{j}}{w_{i}}\right)  \left(  \frac{1}{a_{ij}%
}-\frac{v_{j}}{v_{i}}\right)  <0,
\]
implying, from (\ref{mod}),
\[
a_{ij}=\frac{1}{2}\left(  \frac{w_{i}}{w_{j}}+\frac{v_{i}}{v_{j}}\right)
\text{ and }\frac{1}{a_{ij}}=\frac{1}{2}\left(  \frac{w_{j}}{w_{i}}%
+\frac{v_{j}}{v_{i}}\right)  .
\]
So%
\begin{align*}
4  &  =\left(  \frac{w_{i}}{w_{j}}+\frac{v_{i}}{v_{j}}\right)  \left(
\frac{w_{j}}{w_{i}}+\frac{v_{j}}{v_{i}}\right) \\
&  \Leftrightarrow2=\frac{w_{i}v_{j}}{w_{j}v_{i}}+\frac{w_{j}v_{i}}{w_{i}%
v_{j}}\Leftrightarrow\frac{w_{i}}{w_{j}}=\frac{v_{i}}{v_{j}}.
\end{align*}
Condition $\frac{w_{i}}{w_{j}}=\frac{v_{i}}{v_{j}},$ for all $i,j\in
\{1,\ldots,n\},$ implies $w$ and $v$ proportional.
\end{proof}

\bigskip

Note from the proof of Lemma \ref{lprop} that (\ref{mod}) holds for a pair
$i,j$ if and only if $\frac{w_{i}}{w_{j}}=\frac{v_{i}}{v_{j}}.$

\bigskip

We close this section with a topological property of $\mathcal{E}(A)$.

\begin{theorem}
For any $A\in\mathcal{PC}_{n},$ $\mathcal{E}(A)$ is a closed set.
\end{theorem}

\begin{proof}
We verify this by showing that the inefficient vectors, in the complementary
of $\mathcal{E}(A),$ form an open set, by appealing to Theorem \ref{blanq}.
Suppose that $v\notin\mathcal{E}(A),$ then the graph $G(A,v)$ is not strongly
connected. Let $\widetilde{v}$ be a sufficiently small perturbation of $v$
(i.e. $\widetilde{v}$ lies in an open ball about $v,$ whose radius is
positive, but as small as we like). Then, if $i\rightarrow j$ is not an edge
of $G(A,v),$ then it is not an edge of $G(A,\widetilde{v}).$ Then
$G(A,\widetilde{v})$ has no more edges (under inclusion) than $G(A,v).$ Since
the latter was not strongly connected, the former also is not, so that
$\widetilde{v}\notin\mathcal{E}(A).$
\end{proof}

\bigskip

We also note that, if $w\in\mathcal{E}(A)$ and the matrix $D(A,w)$ has no $0$
off-diagonal entries$,$ then $\widetilde{w}\in\mathcal{E}(A)$ for any
sufficiently small perturbation $\widetilde{w}$ of $w,$ and $w\in
\mathcal{E}(\widetilde{A})$ for any sufficiently small reciprocal perturbation
$\widetilde{A}$ of $A.$

\section{Inductive construction of efficient vectors\label{s4}}

\bigskip Suppose that $A\in\mathcal{PC}_{n}$ and that $w\in\mathcal{E}(A(n)).$
Then $G(A(n),w)$ is strongly connected. May $w$ be extended to an efficient
vector for $A,$ and, if so, how? For a positive scalar $x,$ the vector
$w_{x}:=\left[
\begin{array}
[c]{c}%
w\\
x
\end{array}
\right]  \in\mathcal{E}(A)$ if and only if $G\left(  A,w_{x}\right)  $ is
strongly connected. But, since the subgraph induced by vertices $1,2,\ldots
,n-1$ of $G\left(  A,w_{x}\right)  $ is $G(A(n),w)$ and the latter is strongly
connected, $G\left(  A,w_{x}\right)  $ is strongly connected if and only if
there are edges from vertex $n$ to vertices in $G(A(n),w)$ and also edges from
the latter to $n$ (see Proposition 3 in \cite{CFF}). Since the vector of the
first $n-1$ entries of the last column of $D\left(  A,w_{x}\right)  $ is
$\frac{1}{x}w$ less the vector of the first $n-1$ entries of $a_{n}$ (the last
column of $A)$, there are such edges if and only if this difference vector has
a $0$ entry or both positive and negative entries. This means that among
$\frac{w_{i}}{x}-a_{in},$ $i=1,\ldots,n-1,$ there are both nonnegative and
nonpositive numbers. We restate this as

\begin{theorem}
\label{thext}For $A\in\mathcal{PC}_{n}$ and $w\in\mathcal{E}(A(n)),$ the
vector
\[
\left[
\begin{array}
[c]{c}%
w\\
x
\end{array}
\right]  \in\mathcal{E}(A)
\]
if and only if the scalar $x$ satisfies
\[
x\in\left[  \min_{1\leq i\leq n-1}\frac{w_{i}}{a_{in}},\quad\max_{1\leq i\leq
n-1}\frac{w_{i}}{a_{in}}\right]  .
\]

\end{theorem}

Of course, the above interval is nonempty. This leads to a natural algorithm
to construct a large subset of $\mathcal{E}(A)$ for $A\in\mathcal{PC}_{n}.$

Choose the upper left $2\times2$ principal submatrix $A[\{1,2\}]$ of $A.$ It
is consistent and, up to a factor of scale, has only one efficient vector
$w[\{1,2\}].$ Now extend this vector, in all possible ways, to an efficient
vector for $A[\{1,2,3\}],$ according to Theorem \ref{thext}. This gives the
set $w[\{1,2,3\}]\subseteq\mathcal{E}(A[\{1,2,3\}]).$ Now, continue extending
each vector in $w[\{1,2,3\}]$ to an element of $\mathcal{E}(A[\{1,2,3,4\}])$
in the same way, and so on. This terminates in a subset $w[\{1,2,\ldots
,n\}]\subseteq\mathcal{E}(A).$

We make two important observations. First, we may instead start with some
other $2$-by-$2$ principal submatrix $A[\{i,j\}],$ $i\neq j,$ and proceed
similarly, either by inserting the new entry of the next efficient vector in
the appropriate position, or by placing $A[\{i,j\}]$ in the upper left
$2$-by-$2$ submatrix, via permutation similarity, and proceeding in exactly
the same way. We note that starting in two different positions may produce
different terminal sets (Example \ref{ex1}), and the union of all possible
terminal sets is contained in $\mathcal{E}(A).$

Second, $w[\{1,2,\ldots,n\}]$ may be a proper subset of $\mathcal{E}(A),$ as
truncation of a vector (deletion of an entry) from an efficient vector for $A$
may not give an efficient vector for the corresponding principal submatrix
(see Example \ref{ex1}).

\begin{example}
\label{ex1}Let%
\[
A=\left[
\begin{array}
[c]{ccc}%
1 & 1 & \frac{3}{2}\\
1 & 1 & 1\\
\frac{2}{3} & 1 & 1
\end{array}
\right]  .
\]
The efficient vectors for $A[\{1,2\}]$ are proportional to%
\[
\left[
\begin{array}
[c]{cc}%
1 & 1
\end{array}
\right]  ^{T}.
\]
By Theorem \ref{thext}, the vectors of the form
\[
\left[
\begin{array}
[c]{ccc}%
1 & 1 & w_{3}%
\end{array}
\right]  ^{T}%
\]
with
\[
\min\left\{  \frac{2}{3},1\right\}  \leq w_{3}\leq\max\left\{  \frac{2}%
{3},1\right\}  \quad\Leftrightarrow\quad\frac{2}{3}\leq w_{3}\leq1
\]
are efficient for $A$ (and, of course, all positive vectors proportional to
them)$.$

The efficient vectors for $A[\{1,3\}]$ are proportional to%
\[
\left[
\begin{array}
[c]{cc}%
\frac{3}{2} & 1
\end{array}
\right]  ^{T}.
\]
By Theorem \ref{thext}, the vectors of the form
\[
\left[
\begin{array}
[c]{ccc}%
\frac{3}{2} & w_{2} & 1
\end{array}
\right]  ^{T}%
\]
with
\[
\min\left\{  \frac{3}{2},1\right\}  \leq w_{2}\leq\max\left\{  \frac{3}%
{2},1\right\}  \quad\Leftrightarrow\quad1\leq w_{2}\leq\frac{3}{2}%
\]
are efficient for $A.$

The efficient vectors for $A[\{2,3\}]$ are proportional to%
\[
\left[
\begin{array}
[c]{cc}%
1 & 1
\end{array}
\right]  ^{T}.
\]
By Theorem \ref{thext}, the vectors of the form
\[
\left[
\begin{array}
[c]{ccc}%
w_{1} & 1 & 1
\end{array}
\right]  ^{T},
\]
with
\[
\min\left\{  1,\frac{3}{2}\right\}  \leq w_{1}\leq\max\left\{  1,\frac{3}%
{2}\right\}  \quad\Leftrightarrow\quad1\leq w_{1}\leq\frac{3}{2},
\]
are efficient for $A.$ 

Note that, by Theorem \ref{tmain},
\[
\mathcal{E}(A)=\left\{  \left[
\begin{array}
[c]{ccc}%
w_{1} & w_{2} & w_{3}%
\end{array}
\right]  ^{T}:w_{3}\leq w_{2}\leq w_{1}\leq\frac{3}{2}w_{3}\right\}  .
\]
For example, the vector $\left[
\begin{array}
[c]{ccc}%
\frac{4}{3} & \frac{7}{6} & 1
\end{array}
\right]  ^{T}$ is efficient for $A,$ though it does not belong to the set of
vectors determined above, as no vector obtained from it by deleting one entry
is efficient for the corresponding $2$-by-$2$ principal submatrix.
\end{example}

There are cases in which we know all the efficient vectors for a larger
submatrix and then we can start our building process with this submatrix. In
fact, taking into account Theorem \ref{tmain}, all efficient vectors for a
$3$-by-$3$ reciprocal matrix are known, as such matrix is a simple perturbed
consistent matrix. Thus, it is always possible to start the process from a
$3$-by-$3$ principal submatrix.

\begin{example}
\label{ex2}Consider the matrix
\begin{equation}
A=\left[
\begin{array}
[c]{ccccc}%
1 & 1 & 9 & 4 & \frac{1}{2}\\
1 & 1 & 1 & 4 & 3\\
\frac{1}{9} & 1 & 1 & 1 & \frac{1}{3}\\
\frac{1}{4} & \frac{1}{4} & 1 & 1 & 2\\
2 & \frac{1}{3} & 3 & \frac{1}{2} & 1
\end{array}
\right]  . \label{A}%
\end{equation}
By Theorem \ref{tmain}, the efficient vectors for $A[\{1,2,3\}]$ are the
vectors of the form%
\[
\left[
\begin{array}
[c]{ccc}%
w_{1} & w_{2} & w_{3}%
\end{array}
\right]  ^{T},
\]
with $w_{3}\leq w_{2}\leq w_{1}\leq9w_{3}.$ By Theorem \ref{thext}, the
vectors of the form
\[
\left[
\begin{array}
[c]{cccc}%
w_{1} & w_{2} & w_{3} & w_{4}%
\end{array}
\right]  ^{T},
\]
with
\begin{align*}
\min\left\{  \frac{w_{1}}{4},\frac{w_{2}}{4},w_{3}\right\}   &  \leq w_{4}%
\leq\max\left\{  \frac{w_{1}}{4},\frac{w_{2}}{4},w_{3}\right\}
\Leftrightarrow\\
\min\left\{  \frac{w_{2}}{4},w_{3}\right\}   &  \leq w_{4}\leq\max\left\{
\frac{w_{1}}{4},w_{3}\right\}
\end{align*}
are efficient for $A[\{1,2,3,4\}].$ Again by Theorem \ref{thext}, the vectors
of the form
\[
\left[
\begin{array}
[c]{ccccc}%
w_{1} & w_{2} & w_{3} & w_{4} & w_{5}%
\end{array}
\right]  ^{T}%
\]
with
\begin{align*}
\min\left\{  2w_{1},\frac{w_{2}}{3},3w_{3},\frac{w_{4}}{2}\right\}   &  \leq
w_{5}\leq\max\left\{  2w_{1},\frac{w_{2}}{3},3w_{3},\frac{w_{4}}{2}\right\}
\Leftrightarrow\\
\min\left\{  \frac{w_{2}}{3},\frac{w_{4}}{2}\right\}   &  \leq w_{5}\leq
\max\left\{  2w_{1},3w_{3}\right\}  ,
\end{align*}
are efficient for $A.$ For instance, the vectors%
\[
\left[
\begin{array}
[c]{ccccc}%
3 & 2 & 1 & \frac{3}{4} & w_{5}%
\end{array}
\right]  ^{T}%
\]
with $\frac{3}{8}\leq w_{5}\leq6,$ are efficient for $A.$ Moreover, these are
the only efficient vectors with the given first four entries.
\end{example}

We observe that, if $A\in\mathcal{PC}_{n}$ is a (inconsistent) simple
perturbed consistent matrix, then $A$ has a principal $3$-by-$3$
(inconsistent) simple perturbed consistent submatrix $B.$ If we start the
inductive construction of efficient vectors for $A$ with the submatrix $B,$
for which $\mathcal{E}(B)$ is known by Lemma \ref{lsim} and Theorem
\ref{tmain}, then we obtain $\mathcal{E}(A).$ This fact follows from Corollary
9 in \cite{CFF}, taking into account that, by Lemma \ref{lsim}, and Remark 4.6
in \cite{Fu22}, we may focus on $A=Z_{n}(x),$ for some $x>0.$

Similarly, if $A\in\mathcal{PC}_{n}$ is a double perturbed consistent matrix
(that is, $A\ $is obtained from a consistent matrix by modifying two entries
above the diagonal and the corresponding reciprocal entries), in which no two
perturbed entries lie in the same row and column, then $A$ has a principal
$4$-by-$4$ double perturbed consistent submatrix $B$ of the same type and, by
Theorem 4.2 in \cite{Fu22} and Lemma \ref{lsim}, $\mathcal{E}(B)$ is known. By
Corollary 4.5 in \cite{Fu22}, if we start the inductive construction of
efficient vectors with $B,$ then again we obtain $\mathcal{E}(A).$

Of course, in these simple and double perturbed consistent cases, all the
efficient vectors for $A\in\mathcal{PC}_{n}$ are already known (Theorem
\ref{tmain}, and Theorem 4.2 in \cite{Fu22}).

\section{Columns of a reciprocal matrix\label{s5}}

Previously, it has been noted (Theorem \ref{theogmallcolumns}) that the
Hadamard geometric mean of all columns of $A\in\mathcal{PC}_{n}$ is efficient
for $A$. Interestingly, each individual column of $A$ is efficient.

\begin{lemma}
\label{lcol}Let $A\in\mathcal{PC}_{n}.$ Then any column of $A$ lies in
$\mathcal{E}(A).$
\end{lemma}

\begin{proof}
Let $a_{j}$ be the $j$-th column of $A.$ Then the $j$-th column of
$D(A,a_{j})$ has entries $\frac{a_{ij}}{1}-a_{ij}=0$. Hence, $G(A,a_{j})$ has
an undirected star on $n$ vertices as an induced subgraph and is, therefore,
strongly connected, verifying that $a_{j}$ is efficient, by Theorem
\ref{blanq}.
\end{proof}

\bigskip

Further, the geometric mean of any subset of the columns of a reciprocal
matrix $A$ also lies in $\mathcal{E}(A).$ To prove this result, we use the
following lemma.

\begin{lemma}
\label{col1}Let $A\in\mathcal{PC}_{n},$ $D=\operatorname*{diag}(d_{1}%
,\ldots,d_{n})$ be a positive diagonal matrix and $1\leq s\leq n.$ If $w$ is
the geometric mean of $s$ columns of $A$ then $Dw$ is a positive multiple of
the geometric mean of the corresponding $s$ columns of $DAD^{-1}.$
\end{lemma}

\begin{proof}
By a possible permutation similarity and taking into account Lemma \ref{lsim},
suppose, without loss of generality, that $w$ is the geometric mean of the
first $s$ columns of $A$. (Note that, if $w$ is the geometric mean of a
reciprocal matrix $A$ then $Pw$ is the geometric mean of $PAP^{T}$ for a
permutation matrix $P.$) The $i$-th entry of $Dw$ is $d_{i}\Pi_{j=1}^{s}%
a_{ij}^{\frac{1}{s}}$. On the other hand, the $i$-th entry of the geometric
mean $v$ of the first $s$ columns of $DAD^{-1}$ is $\Pi_{j=1}^{s}\left(
\frac{d_{i}a_{ij}}{d_{j}}\right)  ^{\frac{1}{s}}=d_{i}\Pi_{j=1}^{s}\left(
\frac{a_{ij}}{d_{j}}\right)  ^{\frac{1}{s}}.$ Thus, the quotient of the $i$-th
entries of $Dw$ and $v$ is $\Pi_{j=1}^{s}(d_{j})^{\frac{1}{s}},$ which does
not depend on $i,$ implying the claim.
\end{proof}

\begin{theorem}
\label{thcol}Let $A\in\mathcal{PC}_{n}.$ Then the geometric mean of any
collection of distinct columns of $A$ lies in $\mathcal{E}(A).$
\end{theorem}

\begin{proof}
Let $1\leq s\leq n.$ We show that the geometric mean $w_{A}$ of $s$ distinct
columns of $A$ is efficient for $A.$ The proof is by induction on $n$. For
$n=2,$ the result is straightforward. Suppose that $n>2.$ If $s=1$ or $s=n,$
the result follows from Lemma \ref{lcol} or Theorem \ref{theogmallcolumns},
respectively. Suppose that $1<s<n.$ By Lemmas \ref{lsim} and \ref{col1}, we
may and do assume that the $s$ columns of $A$ are the first ones, and the
entries in the last column and in the last row of $A$ are all equal to $1$,
that is,
\[
A=\left[
\begin{array}
[c]{cc}%
B & \mathbf{e}\\
\mathbf{e}^{T} & 1
\end{array}
\right]  ,
\]
where $\mathbf{e}$ is the $(n-1)$-vector with all entries equal to $1$ and
$B\in\mathcal{PC}_{n-1}$. Let $b_{1},\ldots,b_{s}$ be the first $s$ columns of
$B,$ so that, for $j=1,\ldots,s,$
\[
a_{j}=\left[
\begin{array}
[c]{c}%
b_{j}\\
1
\end{array}
\right]  .
\]
We have%
\[
D(A,w_{A})=\left[
\begin{array}
[c]{cc}%
D(B,w_{B}) & w_{B}-\mathbf{e}\\
w_{B}^{(-T)}-\mathbf{e}^{T} & 1
\end{array}
\right]  ,
\]
in which $w_{B}$ is the geometric mean of columns $b_{1},\ldots,b_{s}.$ By the
induction hypothesis, $w_{B}$ is efficient for $B$ so that, by Theorem
\ref{blanq}, $G(B,w_{B})$ is strongly connected. Thus, $G(A,w_{A})$ is
strongly connected if and only if there are edges from vertex $n$ to vertices
in $G(B,w_{B})$ and also vertices from the latter to $n$ (see the observation
before Theorem \ref{thext}). Then, $G(A,w_{A})$ is strongly connected if and
only if $w_{B}-\mathbf{e}$ is neither strictly positive nor negative. The
product of the first $s$ entries of $w_{B}$ is $l^{\frac{1}{s}}$, where $l$ is
the product of the entries of $B[\{1,\ldots,s\}].$ Since this matrix is
reciprocal, then $l=1.$ Thus, the vector formed by the first $s$ entries of
$w_{B}$ is neither strictly greater than $1$ nor strictly less than $1$,
implying that $G(A,w_{A})$ is strongly connected.
\end{proof}

\bigskip We observe that we have $2^{n}-1$ (nonempty) distinct subsets of
columns of $A\in\mathcal{PC}_{n}$ (not necessarily corresponding to different
geometric means).

\bigskip The sets of efficient vectors for matrices in $\mathcal{PC}_{2}$ (any
$2$-by-$2$ reciprocal matrix is consistent) and in $\mathcal{PC}_{3}$ (any
$3$-by-$3$ reciprocal matrix is a simple perturbed consistent matrix) are
closed under geometric means. In the latter case, this follows from Lemma
\ref{lsim} and the facts that a matrix in $\mathcal{PC}_{3}$ is monomial
similar to $Z_{3}(x),$ for some $x>0$, and, by Theorem \ref{tmain}, the set of
efficient vectors for $Z_{3}(x)$ is closed under geometric mean$.$ However,
the set of efficient vectors for matrices in $\mathcal{PC}_{n},$ with $n>3,$
may not be closed under geometric mean, as the next example illustrates.

\begin{example}
\label{exgm}Let $A$ be the matrix in (\ref{A}). Let
\[
A^{\prime}=A(5)=\left[
\begin{array}
[c]{cccc}%
1 & 1 & 9 & 4\\
1 & 1 & 1 & 4\\
\frac{1}{9} & 1 & 1 & 1\\
\frac{1}{4} & \frac{1}{4} & 1 & 1
\end{array}
\right]  ,
\]
the $4$-by-$4$ principal submatrix of $A$ obtained by deleting the $5$-th row
and column. Taking into account Example \ref{ex2}, the vectors
\[
w=\left[
\begin{array}
[c]{c}%
4.1\\
4.1\\
1\\
1
\end{array}
\right]  \text{ and }v=\left[
\begin{array}
[c]{c}%
4.2\\
4\\
3\\
1
\end{array}
\right]
\]
are efficient for $A^{\prime}.$ However the vector $\left(  w\circ v\right)
^{\left(  \frac{1}{2}\right)  }$ is not efficient for $A$ as the first three
entries of the last column of $D(A^{\prime},\left(  w\circ v\right)  ^{\left(
\frac{1}{2}\right)  })$ are positive and, therefore, $G(A^{\prime},(w\circ
v)^{\left(  \frac{1}{2}\right)  })$ is not strongly connected.$\allowbreak$
\end{example}

\section{Numerical experiments\label{numerical}}

We next give numerical examples in which we compare the geometric means $w$ of
the vectors in different proper subsets of columns of a reciprocal matrix $A$
with the geometric mean of all columns of $A,$ denoted here by $w_{C},$ a
vector proposed by several authors to obtain a consistent matrix approximating
$A$, as this method has a strong axiomatic background \cite{barzilai,
csato2018, csato2019, fichtner86, lundy}. Recall from Section \ref{s5} that
all these vectors are efficient for $A.$ We take $\left\Vert D(A,w)\right\Vert
_{1},$ the sum of all entries of $\left\vert D(A,w)\right\vert ,$ as a measure
of effectiveness of $w\in\mathcal{E}(A),$ as well as $\left\Vert
D(A,w)\right\Vert _{2},$ the Frobenius norm of $D(A,w).$ Recall that, for an
$n$-by-$n$ matrix $B=[b_{ij}],$ we have%
\[
\left\Vert B\right\Vert _{2}=\left(  \sum_{i,j=1,\ldots,n}(b_{ij})^{2}\right)
^{\frac{1}{2}}.
\]
Other measures (norms) are possible. For comparison, we also consider the case
in which $w$ is the Perron eigenvector of $A,$ denoted by $w_{P}$, as it is
one of the most used vectors to estimate a consistent matrix close to $A$. Our
experiments were done using the software Octave version 6.1.0.

\begin{example}
\label{exx11}Consider the matrix%
\[
A=\left[
\begin{array}
[c]{ccccc}%
1 & \frac{9}{5} & \frac{6}{5} & 12 & 6\\
&  &  &  & \\
\frac{5}{9} & 1 & \frac{4}{5} & 100 & 5\\
&  &  &  & \\
\frac{5}{6} & \frac{5}{4} & 1 & \frac{17}{10} & 6\\
&  &  &  & \\
\frac{1}{12} & \frac{1}{100} & \frac{10}{17} & 1 & 3\\
&  &  &  & \\
\frac{1}{6} & \frac{1}{5} & \frac{1}{6} & \frac{1}{3} & 1
\end{array}
\right]  \in\mathcal{PC}_{5}.
\]
There are $31$ distinct subsets of the set of columns of $A.$ We identify each
subset with a sequence of five $0/1$ numbers, in which a $1$ in position $i$
means that the $i$-th column of $A$ belongs to the subset, while a $0$ means
that it does not belong to the subset. The sequences are in increasing
(numerical) order and by $S_{i}$ we denote the subset of columns associated
with the $i$-th sequence. Note that $S_{31}$ is the set of all columns of $A.$
By $w_{i}$ we denote the geometric mean of the vectors in $S_{i}.$

In Table \ref{tab1} we give the norms $\left\Vert D(A,w_{i})\right\Vert _{1}$
and $\left\Vert D(A,w_{i})\right\Vert _{2},$ $i=1,\ldots,31.$ In Table
\ref{tab2} we emphasize the results obtained for the geometric mean $w_{C}$ of
all columns$,$ for the vectors that produce the smallest and the largest
values of $\left\Vert D(A,w_{i})\right\Vert _{1}$ and $\left\Vert
D(A,w_{i})\right\Vert _{2},$ and also consider the case of the Perron
eigenvector $w_{P}$ of $A$ (which is efficient). Note that
\[
\frac{\max_{i}\left\Vert D(A,w_{i})\right\Vert _{1}}{\min_{i}\left\Vert
D(A,w_{i})\right\Vert _{1}}=3.7048\qquad\text{and}\qquad\frac{\max
_{i}\left\Vert D(A,w_{i})\right\Vert _{2}}{\min_{i}\left\Vert D(A,w_{i}%
)\right\Vert _{2}}=5.8235.
\]
It can be observed that, according to the considered measures, there are
proper subsets of columns that produce better results than those for the
Perron eigenvector and for the set of all columns.

We summarize our results in Figure \ref{fig1}, in which we give a graphic with
a comparison of all the results obtained. In the $x$ axis we have the index
$i$ of each subset $S_{i}$ of columns. In the $y$ axis we have the values of
$\left\Vert D(A,w_{i})\right\Vert _{1}$ and $\left\Vert D(A,w_{i})\right\Vert
_{2}$ for the different vectors $w_{i}.$ A line jointing the values of each of
these norms for the different subsets of columns is plotted. A horizontal line
corresponding to each of the considered norms for the Perron eigenvector also appears.
\end{example}

\bigskip

\begin{example}
\label{exx22}Consider the matrix%
\[
A=\left[
\begin{array}
[c]{cccccccc}%
1 & \frac{9}{5} & \frac{6}{5} & 12 & 6 & 2 & 5 & 3\\
&  &  &  &  &  &  & \\
\frac{5}{9} & 1 & \frac{4}{5} & \frac{1}{10} & 6 & \frac{23}{10} & \frac{1}{2}
& \frac{43}{10}\\
&  &  &  &  &  &  & \\
\frac{5}{6} & \frac{5}{4} & 1 & \frac{17}{10} & 6 & \frac{1}{5} & 50 &
\frac{3}{10}\\
&  &  &  &  &  &  & \\
\frac{1}{12} & 10 & \frac{10}{17} & 1 & 3 & 12 & 25 & \frac{13}{10}\\
&  &  &  &  &  &  & \\
\frac{1}{6} & \frac{1}{6} & \frac{1}{6} & \frac{1}{3} & 1 & \frac{21}{10} &
2 & 3\\
&  &  &  &  &  &  & \\
\frac{1}{2} & \frac{10}{23} & 5 & \frac{1}{12} & \frac{10}{21} & 1 & 1 & 3\\
&  &  &  &  &  &  & \\
\frac{1}{5} & 2 & \frac{1}{50} & \frac{1}{25} & \frac{1}{2} & 1 & 1 & 3\\
&  &  &  &  &  &  & \\
\frac{1}{3} & \frac{10}{43} & \frac{10}{3} & \frac{10}{13} & \frac{1}{3} &
\frac{1}{3} & \frac{1}{3} & 1
\end{array}
\right]  \in\mathcal{PC}_{8}.
\]
Table \ref{tab3} and Figure \ref{fig2} are the analogs of Table \ref{tab2} and
Figure \ref{fig1} for the $8$-by-$8$ reciprocal matrix considered here. Note
that in this case we have $255$ different subsets $S_{i}$ of the set of
columns of $A.$ Again, a proper subset of the columns produces better results
than either all columns or the Perron vector (which is efficient for $A$).
Note that
\[
\frac{\max_{i}\left\Vert D(A,w_{i})\right\Vert _{1}}{\min_{i}\left\Vert
D(A,w_{i})\right\Vert _{1}}=5.0432\qquad\text{and}\qquad\frac{\max
_{i}\left\Vert D(A,w_{i})\right\Vert _{2}}{\min_{i}\left\Vert D(A,w_{i}%
)\right\Vert _{2}}=10.890,
\]

\end{example}

\bigskip

\begin{remark}
\label{exx3}In this example we generated $100$ reciprocal matrices
$A\in\mathcal{PC}_{5},$ in which the entries in the upper triangular part are
random real numbers in the interval $(0,100),$ and compare, in terms of the
1-norm (Figure \ref{fig3}) and the Frobenius norm (Figure \ref{fig4}) of
$D(A,w),$ the cases in which $w$ is the geometric mean $w_{C}$ of all columns
of $A$ and $w$ is the vector that produces the smallest norm among all
geometric means of the subsets of the columns of $A.$ Again, it can be
verified that, in general, a proper subset of columns produce better results.
\end{remark}

\bigskip

In our previous examples we have emphasized that the minimum $1$-norm and the
minimum Frobenius norm of $D(A,w),$ when $w$ runs over the geometric means of
the sets of columns of a reciprocal matrix $A,$ is, in general, not attained
by $w_{C},$ the geometric mean of all columns of $A.$ However, if we consider
a large set of random reciprocal matrices $A_{j}$, we can see that the sum,
for all $A_{j}$'s, of some normalization of $\left\Vert D(A_{j},w_{C}%
\right\Vert _{2}$ performs well when compared to the corresponding sums for
other sets of columns, even when $\min\left\Vert D(A_{j},w\right\Vert _{2},$
with $w$ running over the geometric means of the subsets of columns of
$A_{j},$ is not attained by $w_{C}$ for most $j$'s.

\begin{remark}
\label{ex18}We consider reciprocal matrices $A_{j}\in\mathcal{PC}_{5}$,
$j=1,\ldots,1000,$ by generating $1000$ matrices $B_{j}$ with random real
entries in the interval $(0,10)$ and letting $A_{j}=B_{j}\circ(B_{j}%
^{(-1)})^{T},$ where $B_{j}^{(-1)}$ denotes the entry-wise inverse of $B_{j}$
and $\circ$ the Hadamard product. For each $i=1,\ldots,31,$ we determine
\[
p(i):=%
{\displaystyle\sum\limits_{j=1}^{1000}}
\frac{\left\Vert D(A_{j},w_{ij})\right\Vert _{2}}{\min_{k}\left\Vert
D(A_{j},w_{kj})\right\Vert _{2}},
\]
in which $w_{ij}$ is the geometric mean of subset $S_{i}$ of the columns of
$A_{j}.$ (We identify the subsets $S_{i}$ with indices of columns, as
introduced in Example \ref{exx11}.) In Table \ref{tab4} we display the values
of $p(i),$ $i=1,\ldots,31.$ For each $i,$ we also display the number $n(i)$ of
$j$'s for which $\min\left\Vert D(A_{j},w\right\Vert _{2}$, when $w$ runs over
all the geometric means of the subsets of columns of $A_{j}$, is attained by
the subset $S_{i}$.
\end{remark}

\bigskip

We finally give an example illustrating that, close to consistency, all
subsets of columns perform about the same, as expected.

\begin{example}
Consider the matrix%
\[
A=\left[
\begin{array}
[c]{ccccc}%
1 & 1 & 1 & 1.1 & 1\\
1 & 1 & 0.9 & 1 & 1\\
1 & \frac{1}{0.9} & 1 & 1 & 1.1\\
\frac{1}{1.1} & 1 & 1 & 1 & 1\\
1 & 1 & \frac{1}{1.1} & 1 & 1
\end{array}
\right]  \in\mathcal{PC}_{5}.
\]
If $w_{i}$, $i=1,\ldots,31,$ are the geometric means of the $31$ subsets of
the columns of $A,$ we have
\[
\frac{\max_{i}\left\Vert D(A,w_{i})\right\Vert _{1}}{\min_{i}\left\Vert
D(A,w_{i})\right\Vert _{1}}=1.7121\qquad\text{and}\qquad\frac{\max
_{i}\left\Vert D(A,w_{i})\right\Vert _{2}}{\min_{i}\left\Vert D(A,w_{i}%
)\right\Vert _{2}}=1.9125,
\]
which are small when compared with the corresponding values in Examples
\ref{exx11} and \ref{exx22}. Also,%
\[
\frac{\left\Vert D(A,w_{C})\right\Vert _{1}}{\min_{i}\left\Vert D(A,w_{i}%
)\right\Vert _{1}}=1.0136\qquad\text{and}\qquad\frac{\left\Vert D(A,w_{C}%
)\right\Vert _{2}}{\min_{i}\left\Vert D(A,w_{i})\right\Vert _{2}}=1.
\]

\end{example}

\section{Conclusions\label{s6}$\allowbreak$}

In the context of the Analytic Hierarchic Process, pairwise comparison
matrices (PC matrices), also called reciprocal matrices, appear to rank
different alternatives. In practice, the obtained reciprocal matrices are
usually inconsistent and a good consistent matrix approximating the reciprocal
matrix should be obtained. A consistent matrix is uniquely determined by a
positive vector (the vector of priorities or weights). Many methods have been
proposed in the literature to obtain the vectors from which a consistent
matrix approximating a given reciprocal matrix is constructed. Some of the
most used methods consist on the choice of the Perron eigenvector of the
reciprocal matrix or on the Hadamard geometric mean of all its columns. An
important property that should be satisfied by the vectors on which such a
consistent matrix is based is efficiency. It is known that the Hadamard
geometric mean of all the columns of a reciprocal matrix is efficient, though
the Perron eigenvector not always satisfies this property.

Here we give an algorithm to construct efficient vectors for a reciprocal
matrix $A$ from efficient vectors for principal submatrices of $A.$ We also
show that the geometric mean of the vectors in any nonempty subset of the
columns of $A$ is efficient for $A.$ We give an example that the geometric
mean of two efficient vectors need not be efficient. It leaves the question of
when the Hadamard geometric mean of two efficient vectors is efficient.

We give numerical examples comparing the geometric means obtained from proper
subsets of the columns of $A$ with the geometric mean of all the columns of
$A.$ We conclude that the geometric mean of some proper subsets of columns may
produce better results, also when compared with the Perron eigenvector, which
we include for completeness. So, according to our results, there is no
evidence that the geometric mean of all columns has a universal advantage over
the geometric means of other sets of columns, specially when the level of
inconsistency is significant.

\bigskip\bigskip

\textbf{Declaration} All authors declare that they have no conflicts of interest.

\newpage%

\begin{table}[] \centering
\begin{tabular}
[c]{|c|c|c|c||c|c|c|c|}\hline
$i$ & $S_{i}$ & $\left\Vert D(A,w_{i})\right\Vert _{1}$ & $\left\Vert
D(A,w_{i})\right\Vert _{2}$ & $i$ & $S_{i}$ & $\left\Vert D(A,w_{i}%
)\right\Vert _{1}$ & $\left\Vert D(A,w_{i})\right\Vert _{2}$\\\hline\hline
$1$ & $00001$ & $111.06$ & $98.846$ & $17$ & $10001$ & $111.86$ &
$96.992$\\\hline
$2$ & $00010$ & $401.53$ & $302.311$ & $18$ & $10010$ & $127.49$ &
$79.599$\\\hline
$3$ & $00011$ & $150.47$ & $94.324$ & $19$ & $10011$ & $123.22$ &
$90.829$\\\hline
$4$ & $00100$ & $112.16$ & $99.154$ & $20$ & $10100$ & $111.92$ &
$97.300$\\\hline
$5$ & $00101$ & $111.58$ & $99.005$ & $21$ & $10101$ & $111.90$ &
$97.910$\\\hline
$6$ & $00110$ & $152.21$ & $95.328$ & $22$ & $10110$ & $123.95$ &
$91.500$\\\hline
$7$ & $00111$ & $130.39$ & $95.576$ & $23$ & $10111$ & $120.39$ &
$94.442$\\\hline
$8$ & $01000$ & $318.27$ & $209.361$ & $24$ & $11000$ & $154.93$ &
$88.719$\\\hline
$9$ & $01001$ & $115.66$ & $88.581$ & $25$ & $11001$ & $113.28$ &
$90.595$\\\hline
$10$ & $01010$ & $111.78$ & $51.913$ & $26$ & $11010$ & $119.47$ &
$64.637$\\\hline
$11$ & $01011$ & $116.87$ & $76.601$ & $27$ & $11011$ & $112.64$ &
$82.586$\\\hline
$12$ & $01100$ & $116.35$ & $89.645$ & $28$ & $11100$ & $113.69$ &
$91.184$\\\hline
$13$ & $01101$ & $108.38$ & $94.171$ & $29$ & $11101$ & $108.57$ &
$94.052$\\\hline
$14$ & $01110$ & $118.62$ & $78.223$ & $30$ & $11110$ & $113.54$ &
$83.474$\\\hline
$15$ & $01111$ & $115.57$ & $88.454$ & $31$ & $11111$ & $111.96$ &
$89.539$\\\hline
$16$ & $10000$ & $109.61$ & $93.772$ &  &  &  & \\\hline
\end{tabular}
$\ \ \ \ \ \ \ \ $%
\caption{$1$-norm and Frobenius norm of  $D(A,w_i)$, when  $w_i$ is the geometric mean of the columns of  $A$ in the subset $S_i$,  $i=1,...,31$  (Example \ref{exx11}).}\label{tab1}%
\end{table}%
%

\begin{table}[] \centering
\begin{tabular}
[c]{|c||c|c|c|}\hline
& $S_{i}$ associated with $w$ & $\left\Vert D(A,w)\right\Vert _{1}$ &
$\left\Vert D(A,w)\right\Vert _{2}$\\\hline\hline
$w$ st. $\min_{i}\left\Vert D(A,w_{i})\right\Vert _{1}=\left\Vert
D(A,w)\right\Vert _{1}$ & $S_{13}$ & $108.38$ & \\\hline
$w$ st. $\max_{i}\left\Vert D(A,w_{i})\right\Vert _{1}=\left\Vert
D(A,w)\right\Vert _{1}$ & $S_{2}$ & $401.53$ & \\\hline
$w$ st. $\min_{i}\left\Vert D(A,w_{i})\right\Vert _{2}=\left\Vert
D(A,w)\right\Vert _{2}$ & $S_{10}$ &  & $51.913$\\\hline
$w$ st. $\max_{i}\left\Vert D(A,w_{i})\right\Vert _{2}=\left\Vert
D(A,w)\right\Vert _{2}$ & $S_{2}$ &  & $302.31$\\\hline
$w_{C}=w_{31}$ & $S_{31}$ & $111.96$ & $89.539$\\\hline
$w_{P}$ &  & $117.43$ & $84.454$\\\hline
\end{tabular}
$\ \ \ \ \ \ \ \ $%
\caption{Comparison of the performance of the geometric mean  $w_C$ of all  the columns of  $A$, the Perron eigenvector  $w_P$ of  $A$ and the geometric means of the subsets of  the columns of  $A$ with best and worst  behaviors (Example \ref{exx11}).}\label{tab2}%
\end{table}%
%

\begin{figure}[h]
 \includegraphics[width=\linewidth]{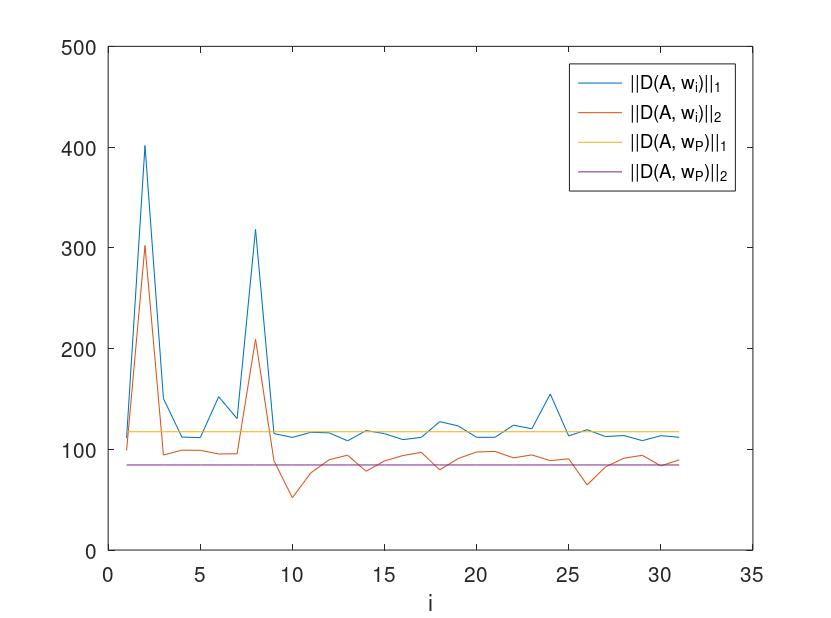}
 \caption{\label{fig1} Comparison of the performance of the geometric means of the subsets
of columns of $A$ and the Perron eigenvector $w_{P}$ of $A$ (Example
\ref{exx11}).}

 \end{figure}
%

\begin{table}[] \centering
\begin{tabular}
[c]{|c||c|c|c|}\hline
& $S_{i}$ associated with $w$ & $\left\Vert D(A,w)\right\Vert _{1}$ &
$\left\Vert D(A,w)\right\Vert _{2}$\\\hline\hline
$w$ st. $\min_{i}\left\Vert D(A,w_{i})\right\Vert _{1}=\left\Vert
D(A,w)\right\Vert _{1}$ & $S_{235}$ & $145.76$ & \\\hline
$w$ st. $\max_{i}\left\Vert D(A,w_{i})\right\Vert _{1}=\left\Vert
D(A,w)\right\Vert _{1}$ & $S_{16}$ & $735.07$ & \\\hline
$w$ st. $\min_{i}\left\Vert D(A,w_{i})\right\Vert _{2}=\left\Vert
D(A,w)\right\Vert _{2}$ & $S_{34}$ &  & $32.289$\\\hline
$w$ st. $\max_{i}\left\Vert D(A,w_{i})\right\Vert _{2}=\left\Vert
D(A,w)\right\Vert _{2}$ & $S_{16}$ &  & $351.62$\\\hline
$w_{C}=w_{255}$ & $S_{255}$ & $151.16$ & $53.620$\\\hline
$w_{P}$ &  & $150.54$ & $52.462$\\\hline
\end{tabular}
$\ $%
\caption{Comparison of the performance of the geometric mean  $w_C$ of all  the columns of  $A$, the Perron eigenvector $w_P$ of  $A$ and the geometric means of the subsets of  the columns of  $A$ with best and worst  behaviors (Example \ref{exx22}).}\label{tab3}%
\end{table}%
\begin{figure}[h]
 \includegraphics[width=\linewidth]{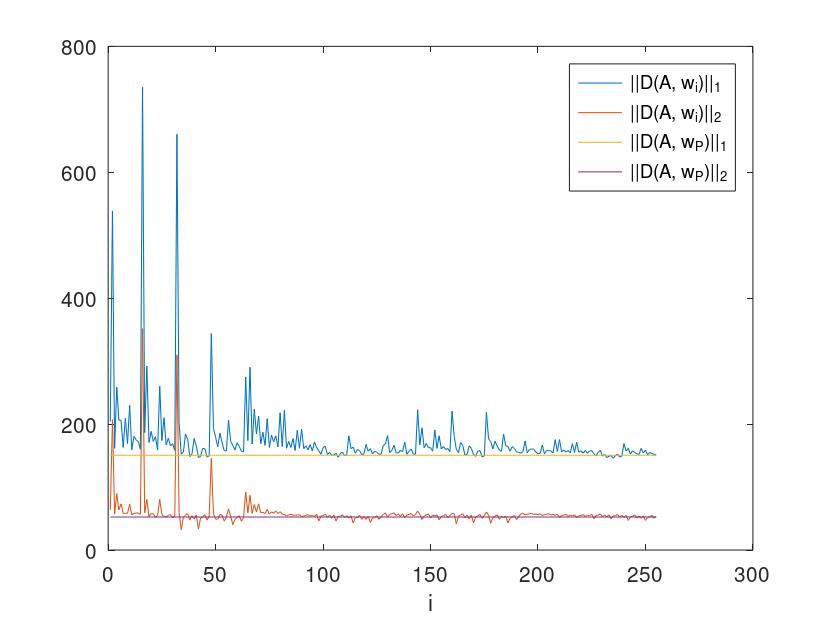}
 \caption{\label{fig2} Comparison of the performance of the geometric means of the subsets of the columns of $A$ and the Perron eigenvector $w_{P}$ of $A$ (Example \ref{exx22}).}

 \end{figure}
\begin{figure}[h]
 \includegraphics[width=\linewidth]{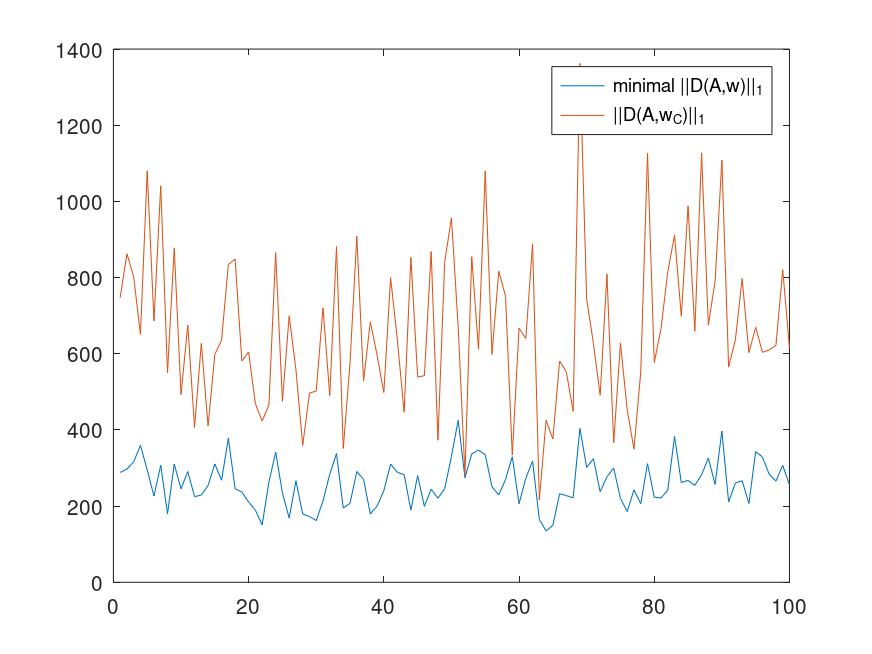}
 \caption{\label{fig3} Comparison of the minimal value of the $1$-norm of $D(A,w)$, when $w$ runs over the geometric means of the subsets of the columns of $A$, with the $1$-norm of $D(A,w),$ when $w$ is the geometric mean $w_{C}$ of all columns of $A$, for $100$ reciprocal matrices $A$ (Experiment \ref{exx3}).}

 \end{figure}

%

\begin{figure}[h]
 \includegraphics[width=\linewidth]{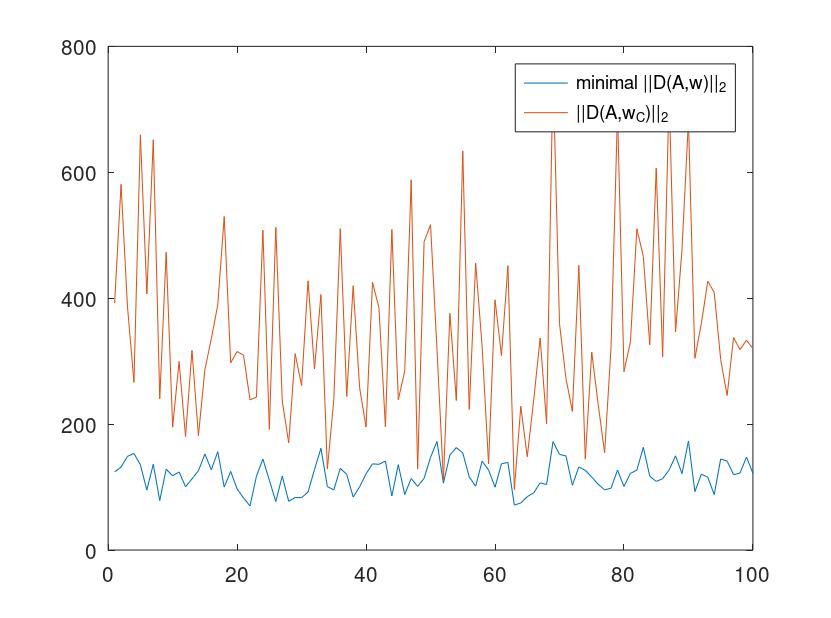}
 \caption{\label{fig4} Comparison of the minimal value of the Frobenius norm of $D(A,w)$, when $w$ runs over the geometric means of the subsets of the columns of $A$, with the Frobenius norm of $D(A,w),$ when $w$ is the geometric mean $w_{C}$ of all columns of $A$, for $100$ reciprocal matrices $A$ (Experiment \ref{exx3}).}

 \end{figure}

%

\begin{table}[] \centering
\begin{tabular}
[c]{|c|c|c|c||c|c|c|c|}\hline
$i$ & $S_{i}$ & $p(i)$ & $n(i)$ & $i$ & $S_{i}$ & $p(i)$ & $n(i)$%
\\\hline\hline
$1$ & $00001$ & $3848.0$ & $4$ & $17$ & $10001$ & $1776.9$ & $50$\\\hline
$2$ & $00010$ & $3839.7$ & $4$ & $18$ & $10010$ & $1770.0$ & $48$\\\hline
$3$ & $00011$ & $1784.1$ & $38$ & $19$ & $10011$ & $1651.0$ & $37$\\\hline
$4$ & $00100$ & $4415.0$ & $3$ & $20$ & $10100$ & $1778.4$ & $44$\\\hline
$5$ & $00101$ & $1763.9$ & $45$ & $21$ & $10101$ & $1636.8$ & $37$\\\hline
$6$ & $00110$ & $1741.5$ & $49$ & $22$ & $10110$ & $1640.7$ & $29$\\\hline
$7$ & $00111$ & $1642.3$ & $42$ & $23$ & $10111$ & $1621.0$ & $24$\\\hline
$8$ & $01000$ & $4015.5$ & $4$ & $24$ & $11000$ & $1770.5$ & $42$\\\hline
$9$ & $01001$ & $1771.0$ & $41$ & $25$ & $11001$ & $1646.4$ & $27$\\\hline
$10$ & $01010$ & $1765.2$ & $31$ & $26$ & $11010$ & $1638.1$ & $36$\\\hline
$11$ & $01011$ & $1650.1$ & $34$ & $27$ & $11011$ & $1624.5$ & $23$\\\hline
$12$ & $01100$ & $1757.7$ & $41$ & $28$ & $11100$ & $1640.4$ & $37$\\\hline
$13$ & $01101$ & $1633.8$ & $48$ & $29$ & $11101$ & $1615.6$ & $32$\\\hline
$14$ & $01110$ & $1621.4$ & $60$ & $30$ & $11110$ & $1612.4$ & $33$\\\hline
$15$ & $01111$ & $1614.2$ & $9$ & $31$ & $11111$ & $1613.8$ & $44$\\\hline
$16$ & $10000$ & $3618.3$ & $4$ &  &  &  & \\\hline
\end{tabular}
$\ \ \ \ \ \ \ \ $%
\caption{For 1000 random $5$-by-$5$-reciprocal matrices $A_j$, $p(i)$ is the sum of the normalized Frobenius norm of  $D(A_j,w_{ij})$, $j=1,...,1000$, in which $w_{ij}$ is the geometric mean of the columns of  $A_j$ in the subset $S_i$;  $n(i)$ is the number of  matrices $A_j$ for which the minimum of  the Frobenius norm of  $D(A_j,w)$ occurs for the subset  $S_i$ (Experiment \ref{ex18}).}\label{tab4}%
\end{table}%


\begin{thebibliography}{99}                                                                                               %


\bibitem {p6}K. \'{A}bele-Nagy, S. Boz\'{o}ki, \textit{Efficiency analysis of
simple perturbed pairwise comparison matrices}, Fundamenta Informaticae 144
(2016), 279-289.

\bibitem {p2}K. \'{A}bele-Nagy, S. Boz\'{o}ki, \"{O}. Reb\'{a}k,
\textit{Efficiency analysis of double perturbed pairwise comparison matrices},
Journal of the Operational Research Society 69 (2018), 707-713.

\bibitem {anh}M. Anholcer, J. F\"{u}l\"{o}p, \textit{Deriving priorites from
inconsistent PCM using the network algorithms}, Annals of Operations Research
274 (2019), 57-74.

\bibitem {baj}G. Bajwa, E. U. Choo, W. C. Wedley, \textit{Effectiveness
analysis of deriving priority vectors from reciprocal pairwise comparison
matrices}, Asia-Pacific Journal of Operational Research, 25(3) (2008), 279--299.

\bibitem {barzilai}J. Barzilai, \textit{Deriving weights from pairwise
comparison matrices}, Journal of the Operational Research Society 48(12)
(1997), 1226-1232.

\bibitem {blanq2006}R. Blanquero, E. Carrizosa, E. Conde, \textit{Inferring
efficient weights from pairwise comparison matrices}, Mathematical Methods of
Operations Research 64 (2006), 271-284.

\bibitem {bozoki2014}S. Boz\'{o}ki, \textit{Inefficient weights from pairwise
comparison matrices with arbitrarily small inconsistency}, Optimization 63,
1893-1901 (2014).

\bibitem {bozoki2016}S. Boz\'{o}ki, L. Csat\'{o}, J. Temesi, \textit{An
application of incomplete pairwise comparison matrices for ranking top tennis
players}, European Journal of Operational Research 248(1) (2016), 211--218.

\bibitem {european}S. Boz\'{o}ki, J. F\"{u}l\"{o}p, \textit{Efficient weight
vectors from pairwise comparison matrices}, European Journal of Operational
Research 264 (2018), 419-427.

\bibitem {chao}X. Chao, G. Kou, T. Li, Y. Peng, \textit{Jie Ke versus AlphaGo:
A ranking approach using decision making method for large-scale data with
incomplete information}, European Journal of Operational Research, 265(1)
(2018), 239--247.

\bibitem {choo}E. Choo, W. Wedley, \textit{A common framework for deriving
preference values from pairwise comparison matrices}, Computers and Operations
Research, 31(6) (2004), 893--908.

\bibitem {CW85}G. Crawford, C. Williams, \textit{A note on the analysis of
subjective judgment matrices}, Journal of Mathematical Psychology 29(4)
(1985), 387--405.

\bibitem {CFF}H. Cruz, R. Fernandes, S. Furtado, \textit{Efficient vectors for
simple perturbed consistent matrices}, International Journal of Approximate
Reasoning 139 (2021), 54-68.

\bibitem {csato2013}L. Csat\'{o}, \textit{Ranking by pairwise comparisons for
Swiss-system tournaments}, Central European Journal of Operations Research,
21(4) (2013), 783--803.

\bibitem {csato2018}L. Csat\'{o}, \textit{Characterization of the row
geometric mean ranking with a group consensus axiom}, Group Decision and
Negotiation 27(6) (2018), 1011-1027.

\bibitem {csato2019}L. Csat\'{o}, \textit{A characterization of the
Logarithmic Least Squares Method}, European Journal of Operational Research
276(1) (2019), 212-216.

\bibitem {dij}T. K. Dijkstra, \textit{On the extraction of weights from
pairwise comparison matrices}, Central European Journal of Operations
Research, 21(1) (2013), 103-123.

\bibitem {dyer0}J. Dyer, \textit{Remarks on the Analytic Hierarchy Process},
Management Science 36 (3) (1990), 249-258.

\bibitem {dyer}J. Dyer, \textit{A clarification of \textquotedblleft remarks
on the analytic hierarchy process\textquotedblright}, Management Science 36
(1990), 274--275.

\bibitem {FerFur}R. Fernandes, S. Furtado, \textit{Efficiency of the principal
eigenvector of some triple perturbed consistent matrices, }European Journal of
Operational Research 298 (2022), 1007-1015.

\bibitem {fichtner86}J. Fichtner, \textit{On deriving priority vectors from
matrices of pairwise comparisons}, Socio-Economic Planning Sciences 20(6)
(1986), 341-345.

\bibitem {Fu22}S. Furtado, \textit{Efficient vectors for double perturbed
consistent matrices}, Optimization (online, 2022).

\bibitem {golany}B. Golany, M. Kress, \textit{A multicriteria evaluation of
methods for obtaining weights from ratio-scale matrices}, European Journal of
Operational Research, 69 (2) (1993), 210--220.

\bibitem {HJ}R. A. Horn, C. R. Johnson, \textit{Matrix analysis}, Cambridge
University Press, Cambridge, 1985.

\bibitem {is}Alessio Ishizaka, Ashraf Labib, \textit{Review of the main
developments in the analytic hierarchy proces}s, Expert Systems with
Applications 38 (11) (2011), 14336-14345.

\bibitem {johns}C. R. Johnson, W. B. Beine, T. J. Wang, \textit{Right-left
asymmetry in an eigenvector ranking procedure}, Journal of Mathematical
Psychology, 19(1) (1979), 61--64.

\bibitem {Kula}K. Ku\l akowski, J. Mazurek. M. Strada, \textit{On the
similarity between ranking vectors in the pairwise comparison method}, Journal
of the Operational Research Society 73:9 (2022), 2080-2089.

\bibitem {lundy}M. Lundy, S. Siraj, S. Greco, \textit{The mathematical
equivalence of the \textquotedblleft spanning tree\textquotedblright\ and row
geometric mean preference vectors and its implications for preference
analysis,} European Journal of Operational Research, 257(1) (2017), 197--208.

\bibitem {petroczy2021}D. G. Petr\'{o}czy, \textit{An alternative quality of
life ranking on the basis of remittances}, Socio-Economic Planning Sciences
78:101042 (2021).

\bibitem {petrocsato}D. G. Petr\'{o}czy, and L. Csat\'{o}, \textit{Revenue
allocation in Formula One: A pairwise comparison approach}, International
Journal of General Systems 50(3) (2021),243--261.

\bibitem {saaty1977}T. L. Saaty, \textit{A scaling method for priorities in
hierarchical structures}, Journal of Mathematical Psychology 32 (1977), 234--281.

\bibitem {Saaty}T. L. Saaty, \textit{The Analytic Hierarchy Process},
McGraw-Hill, New York, 1980.

\bibitem {saaty2003}T. L. Saaty, \textit{Decision-making with the AHP: Why is
the principal eigenvector necessary}, European Journal of Operational Research
145 (2003), 85--91.

\bibitem {zeleny}M. Zeleny, \textit{Multiple criteria decision making}, (1982) McGraw-Hill.
\end{thebibliography}
\end{document}